\documentclass[12pt,twoside]{amsart}
\usepackage{amsmath}
\usepackage{amssymb}
\usepackage{amscd}
\usepackage{tikz-cd}
\usepackage{xypic}
\usepackage{url}
\usepackage{hyperref}
\xyoption{all}
\title[Una visi\'on local de los grupos finitos]{Una visi\'on local de los grupos finitos}

\author{Jos\'e Cantarero}
\thanks{}
\address{
\hfill\break Centro de Investigaci\'on en Matem\'aticas, A.C. Unidad M\'erida \\
\hfill\break Parque Cient\'ifico y Tecnol\'ogico de Yucat\'an \\ 
\hfill\break Carretera Sierra Papacal--Chuburn\'a Puerto Km 5.5 \\
\hfill\break Sierra Papacal, M\'erida, YUC 97302 \\
\hfill\break Mexico.
\hfill\break {\emph{Email address: }}{\tt cantarero@cimat.mx}}

\newcommand{\ab}{\operatorname{ab}\nolimits}
\newcommand{\Aut}{\operatorname{Aut}\nolimits}
\newcommand{\cc}{\operatorname{cc}\nolimits}

\newcommand{\Co}{\operatorname{Co}\nolimits}

\newcommand{\Dic}{\operatorname{Dic}\nolimits}
\newcommand{\Ext}{\operatorname{Ext}\nolimits}
\newcommand{\Hom}{\operatorname{Hom}\nolimits}

\newcommand{\Inn}{\operatorname{Inn}\nolimits}

\newcommand{\Mor}{\operatorname{Mor}\nolimits}

\newcommand{\Out}{\operatorname{Out}\nolimits}
\newcommand{\Rep}{\operatorname{Rep}\nolimits}
\newcommand{\RV}{\operatorname{RV}\nolimits}
\newcommand{\SD}{\operatorname{SD}\nolimits}
\newcommand{\SL}{\operatorname{SL}\nolimits}

\newcommand{\Tor}{\operatorname{Tor}\nolimits}

\def \F{{\mathbb F}}
\def \G{{\mathcal G}}

\def \N{{\mathbb N}}

\def \Qa{{\mathcal Q}}
\def \Q{{\mathbb Q}}
\def \R{{\mathbb R}}

\def \Z{{\mathbb Z}}

\newcommand{\higherlim}[2]{\displaystyle\setbox1=\hbox{\rm lim}
	\setbox2=\hbox to \wd1{\leftarrowfill} \ht2=0pt \dp2=-1pt
	\setbox3=\hbox{$\scriptstyle{#1}$}
	\def\test{#1}\ifx\test\empty
	\mathop{\mathop{\vtop{\baselineskip=5pt\box1\box2}}}\nolimits^{#2}
	\else
	\ifdim\wd1<\wd3
	\mathop{\hphantom{^{#2}}\vtop{\baselineskip=5pt\box1\box2}^{#2}}_{#1}
	\else
	\mathop{\mathop{\vtop{\baselineskip=5pt\box1\box2}}_{#1}}%
	\nolimits^{#2}
	\fi\fi}
	
	\newcommand{\highercolim}[2]{\displaystyle\setbox1=\hbox{\rm lim}
	\setbox2=\hbox to \wd1{\rightarrowfill} \ht2=0pt \dp2=-1pt
	\setbox3=\hbox{$\scriptstyle{#1}$}
	\def\test{#1}\ifx\test\empty
	\mathop{\mathop{\vtop{\baselineskip=5pt\box1\box2}}}\nolimits^{#2}
	\else
	\ifdim\wd1<\wd3
	\mathop{\hphantom{^{#2}}\vtop{\baselineskip=5pt\box1\box2}^{#2}}_{#1}
	\else
	\mathop{\mathop{\vtop{\baselineskip=5pt\box1\box2}}_{#1}}%
	\nolimits^{#2}
	\fi\fi}

\newcommand{\Ff}{{\mathcal{F}}}
\newcommand{\Ll}{{\mathcal{L}}}
\newcommand{\pcom}{^\wedge_p}

\theoremstyle{plain}

\newtheorem*{teoremaAlperin}{Teorema de fusi\'on de Alperin-Goldschmidt}
\newtheorem*{teoremaCartan}{Teorema de Cartan-Eilenberg}
\newtheorem{theorem}{Teorema}[section]
\newtheorem{proposition}[theorem]{Proposici\'on}
\newtheorem{corollary}[theorem]{Corolario}
\newtheorem{lemma}[theorem]{Lema}

\theoremstyle{definition}
\newtheorem*{property}{Propiedad global}
\newtheorem*{localproperty}{Propiedad local}
\newtheorem{definition}[theorem]{Definici\'on}
\newtheorem{remark}[theorem]{Observaci\'on}
\newtheorem{example}[theorem]{Ejemplo}

\keywords{Subgrupos de Sylow, sistemas de fusi\'on}

\subjclass{20D20, 55R35}

\begin{document}

\begin{abstract}
Este art\'iculo contiene una introducci\'on b\'asica al estudio local de grupos finitos, incluyendo
una breve perspectiva de la teor\'ia de sistemas de fusi\'on y grupos $p$-locales finitos.
\end{abstract}

\maketitle

\section*{Introducci\'on}

Los fen\'omenos de simetr\'ia se pueden modelar matem\'aticamente usando acciones
de grupos. Configuraciones con un n\'umero finito de simetr\'ias abundan en diversas
ciencias, incluyendo matem\'aticas, por lo cual es \'util comprender la estructura 
de los grupos finitos.

Una manera de llevar esto a cabo es intentar clasificarlos. Los grupos finitos se pueden
construir mediante extensiones a partir de grupos finitos simples, as\'i que un primer paso
en esta direcci\'on es la clasificaci\'on de grupos finitos simples, que se conoce. Sin embargo, 
permanece el problema de determinar las posibles extensiones.

Una alternativa es clasificar las familias de grupos finitos que aparecen como simetr\'ias
en el fen\'omeno de inter\'es. Por ejemplo, se pueden clasificar los grupos finitos de reflexi\'on
complejos o los grupos finitos que act\'uan libremente sobre alguna esfera. Otra alternativa es
usar una relaci\'on de equivalencia m\'as d\'ebil que isomorfismo, pero que contenga la informaci\'on
que queremos. Seguiremos este \'ultimo camino, mediante una filosof\'ia de localizaci\'on.

Localizar consiste en reducir un problema a varios problemas locales. En una primera aproximaci\'on, estos 
problemas locales se obtienen fijando un n\'umero primo y reinterpretando el problema original con objetos 
cuya informaci\'on est\'a concentrada en ese primo. Tambi\'en existen localizaciones m\'as generales, por
ejemplo con respecto a una teor\'ia de cohomolog\'ia o a un morfismo.

Los grupos abelianos finitos se pueden analizar usando las localizaciones en primos de la teor\'ia de grupos
abelianos. En la Secci\'on \ref{section:abelianos} introducimos la noci\'on de equivalencia $p$-local entre grupos abelianos finitos basada en estas localizaciones. Gracias a
la descomposici\'on primaria, la localizaci\'on en el primo $p$ de un grupo abeliano es isomorfa a su $p$-subgrupo 
de Sylow y este es el punto de vista que se puede generalizar a grupos finitos no necesariamente abelianos.

El estudio local de un grupo finito $G$ no solo utiliza un $p$-subgrupo de Sylow $S$, sino tambi\'en las $G$--subconjugaciones (ver Definici\'on \ref{subconjugacion})
entre subgrupos de $S$. En la Secci\'on \ref{section:finitos} extendemos la definici\'on de equivalencia $p$-local para grupos
finitos y describimos invariantes $p$-locales, que nos permiten descartar que dos grupos sean $p$-localmente equivalentes, junto con ejemplos.
Tambi\'en estudiamos la relaci\'on entre ser isomorfos y ser $p$-localmente equivalentes para todo primo $p$, probando que estos
dos conceptos coinciden para grupos nilpotentes, pero no en general.

En la Secci\'on \ref{section:fusion} vemos que la informaci\'on $p$-local de un grupo finito $G$ con $p$-subgrupo de Sylow $S$ se
puede resumir en una categor\'ia llamada el sistema de fusi\'on de $G$ relativo a $S$, de tal manera que dos grupos finitos son
$p$-localmente equivalentes si y solo si sus sistemas de fusi\'on son isomorfos en cierto sentido. Esto nos permite encontrar 
nuevos invariantes $p$-locales, aquellos que se puedan construir a partir del sistema de fusi\'on, como por ejemplo la homolog\'ia
y cohomolog\'ia de grupos con coeficientes en el campo $\F_p$ con $p$ elementos.

El resto del art\'iculo est\'a dedicado a relacionar el estudio $p$-local de los grupos finitos con la homotop\'ia de la $p$-completaci\'on
de sus espacios clasificantes. En la Secci\'on \ref{section:clasificante} damos primero una introducci\'on a los espacios clasificantes de grupos
finitos y la $p$-completaci\'on de espacios, para despu\'es hablar sobre la Conjetura de Martino-Priddy (Teorema \ref{MartinoPriddy} en el texto)
que proporciona la relaci\'on deseada. Finalmente, en la Secci\'on \ref{section:grupoplocal} hablamos sobre algunas de las motivaciones de la teor\'ia de 
grupos $p$-locales finitos.

En este art\'iculo proveemos una visi\'on introductoria del estudio $p$-local de grupos finitos que presupone familiaridad con teor\'ia de grupos, pero
no necesariamente con topolog\'ia algebraica. Es por ello que las tres primeras secciones contienen argumentos y ejemplos detallados, 
mientras que las dos \'ultimas son m\'as superficiales. El contenido est\'a basado en un minicurso impartido por el autor, que se
puede consultar en \cite{Can}. El lector interesado en profundizar en estos temas puede leer los art\'iculos panor\'amicos \cite{AO}, \cite{Broto}, \cite{Cr}, dirigirse a los libros \cite{AKO}, \cite{Li} o al art\'iculo \cite{BLO1}.

\subsection*{Agradecimientos}
La elaboraci\'on de este art\'iculo fue apoyada por CONAHCYT en el a\~no 2023 mediante el 
proyecto de Ciencia de Frontera CF-2023-I-2649.

\section{La descomposici\'on primaria de grupos abelianos finitos}
\label{section:abelianos}

Para cada grupo abeliano finito $A$, existe una descomposici\'on
\[ A = \bigoplus_{p \text{ primo}} A_{(p)}, \]
donde cada $A_{(p)}$ es un $p$-grupo abeliano finito y $A_{(q)}$ es el grupo trivial
para cada primo $q$ que no divida al orden de $A$. A esto se le conoce como la descomposici\'on 
primaria de $A$ y al grupo $A_{(p)}$ se le llama la componente $p$-primaria de $A$. Adem\'as, para cada $p$ se tiene
\[ A_{(p)} \cong \bigoplus_{k=1}^r \Z/p^{m_k}, \]
para ciertos $m_k \geq 0$. Esto determina $A$ salvo isomorfismo pues dos grupos abelianos finitos 
son isomorfos si y solo si tienen exactamente los mismos sumandos en esta descomposici\'on salvo
reordenaci\'on. Esto constituye el teorema de clasificaci\'on de grupos abelianos finitos, que 
corresponde al Teorema 6.1 y al Teorema 6.13 en \cite{R}.

\begin{remark}
La componente $p$-primaria de $A$ est\'a compuesta de los elementos de $A$ cuyo orden es una potencia de $p$, pero
tambi\'en la podemos obtener a partir de $A$ mediante
\[ A_{(p)} = A \otimes_{\Z} \Z_{(p)}, \]
donde $\Z_{(p)}$ es la localizaci\'on de $\Z$ en el primo $p$. Usaremos la notaci\'on $A \pcom$ debido a que
\[ A \pcom = A \otimes_{\Z} \Z \pcom \cong A_{(p)}, \]
donde $\Z \pcom$ es el anillo de enteros $p$-\'adicos.
\end{remark}

Esta descomposici\'on nos permite reducir problemas sobre grupos abelianos finitos a
$p$-grupos abelianos finitos, siendo este nuestro primer contexto donde podemos aplicar
la filosof\'ia de localizaci\'on. Por ejemplo, consideremos las siguientes propiedades.

\begin{property}
Los grupos abelianos finitos $A$ y $B$ son isomorfos.
\end{property}

\begin{localproperty}
Sea $p$ un primo. Las componentes $p$-primarias de $A$ y $B$ son isomorfas.
\end{localproperty}

?`C\'omo podemos saber si esta propiedad local es adecuada para analizar la propiedad original? Para empezar,
puesto que $ - \otimes_{\Z} \Z \pcom$ es un funtor, se tiene
\[ A \cong B  \Longrightarrow A \pcom \cong B \pcom . \]
Es decir, la propiedad global implica la propiedad local. El converso no es cierto, por ejemplo los grupos $\Z/2$ y $\Z/2 \oplus \Z/3$ no son isomorfos,
pero
\[ (\Z/2) ^{\wedge}_2 \cong \Z/2 \cong (\Z/2 \oplus \Z/3) ^{\wedge}_2. \]
Sin embargo, si suponemos que $A \pcom \cong B \pcom$ para todo primo $p$, entonces por la clasificaci\'on de grupos abelianos 
finitos se tiene $A \cong B$. 

\begin{proposition}
\label{LocalAGlobalAbeliano}
Dos grupos abelianos finitos $A$ y $B$ son isomorfos si y solo si $A \pcom \cong B \pcom$ para todo primo $p$.
\end{proposition}

Este resultado es satisfactorio pues reduce la certeza de una propiedad a que se cumplan varias propiedades
locales. Se podr\'ia mejorar ligeramente, pues al ser $A$ y $B$ finitos, solo existen un n\'umero finito de
primos que dividen a los \'ordenes de $A$ \'o $B$. As\'i que ser\'ia suficiente comprobar que $A \pcom \cong B \pcom$
para todo primo $p$ que divida el orden de $A$ o el orden de $B$. La Proposici\'on \ref{LocalAGlobalAbeliano} 
nos inspira a introducir la siguiente relaci\'on de equivalencia.

\begin{definition}
\label{EquivalenciapLocalAbelianos}
Sean $A$ y $B$ grupos abelianos finitos. Una equivalencia $p$-local de $A$ a $B$ es un isomorfismo $A \pcom \to B \pcom$.
\end{definition} 

Si existe una equivalencia $p$-local entre $A$ y $B$, diremos que son $p$-localmente equivalentes y denotaremos $A \simeq_p B$. 
La Proposici\'on \ref{LocalAGlobalAbeliano} nos dir\'ia entonces 
\[ A \cong B \Longleftrightarrow A \simeq_p B \quad \text{para todo primo $p$.} \]
La noci\'on de equivalencia $p$-local es m\'as d\'ebil que isomorfismo. Los grupos $\Z/2$ y $\Z/2 \oplus \Z/3$ no
son isomorfos, pero son $2$-localmente equivalentes.

\begin{definition}
Sea $F$ una asignaci\'on que a cada grupo abeliano finito $A$ lo env\'ia a un objeto $F(A)$ de cierta categor\'ia. Diremos
que $F$ es un invariante $p$-local si $A \simeq_p B$ implica $F(A) \cong F(B)$. Diremos que es un invariante $p$-local completo
si se tiene que $A \simeq_p B$ si y solo si $F(A) \cong F(B)$. 
\end{definition}

En esta definici\'on la categor\'ia podr\'ia ser discreta, es decir, sin morfismos distintos de las identidades. Esto permite
tener invariantes que toman valores en enteros, por ejemplo. He aqu\'i algunos ejemplos inmediatos de invariantes $p$-locales.

\begin{itemize}
\item La m\'axima potencia de $p$ que divide al orden de $A$.
\item El grupo abeliano $\Hom(\Z/p,A)$.
\item La cardinalidad del conjunto de elementos de orden $p^2$.
\end{itemize}

Similarmente, si $\Qa$ es una propiedad que pueden cumplir los grupos abelianos finitos, diremos que es una propiedad $p$-local
si cada vez que se tenga $A \simeq_p B$, entonces $A$ satisface $\Qa$ si y solo si $B$ satisface $\Qa$. En realidad, una propiedad
$p$-local $\Qa$ corresponde al invariante $p$-local que a cada grupo $A$ le asigna $1$ si satisface $\Qa$ y $0$ en otro caso. Por
ejemplo, las siguientes son propiedades $p$-locales.

\begin{itemize}
\item Tener alg\'un elemento de orden $p$.
\item Tener orden primo relativo con $p$.
\end{itemize}

Los invariantes y propiedades $p$-locales nos sirven para descartar si dos grupos son $p$-localmente equivalentes, pues si $F$ es un
invariante $p$-local, se cumple
\[ F(A) \ncong F(B) \Longrightarrow A \not\simeq_p B. \]
Afortunadamente, gracias a la descomposici\'on primaria, existe un invariante $p$-local completo conocido como los coeficientes de $p$-torsi\'on.
Sea $A$ un grupo abeliano finito tal que 
\[ A \pcom \cong \bigoplus_{k >0} (\Z/p^k)^{n_k}. \]
Los coeficientes de $p$-torsi\'on de $A$ est\'an dados por una tupla finita definida de la siguiente manera inductiva. Sea $k$ el m\'inimo
natural tal que $n_k \neq 0$. Entonces ponemos $n_k$ copias de $p^k$ en la tupla. Repetimos el proceso con los elementos de $\N$ mayores que $k$ 
y en cada paso ponemos las copias correspondientes a la derecha de los elementos ya asentados. Otra manera de expresar la misma informaci\'on es mediante la funci\'on
\begin{align*}
f_A \colon \N & \to \N \cup \{ 0 \},\\
            k & \mapsto n_k ,
\end{align*}
que se\~nala las multiplicidades de los $\Z/p^k$ en la descomposici\'on primaria de $A$. Por ejemplo, los coeficientes de $2$-torsi\'on de $(\Z/2)^3  \oplus \Z/12$ est\'an
dados por la tupla $(2,2,2,2^2)$ o por la funci\'on $\N \to \N \cup \{ 0 \}$ que env\'ia $1$ a $3$, que env\'ia $2$ a $1$ y todos los dem\'as naturales a cero. Por supuesto,
todos los coeficientes de $p$-torsi\'on juntos forman los coeficientes de torsi\'on que forman un invariante completo para isomorfismos de grupos abelianos finitos.

\section{Equivalencias $p$-locales de grupos finitos}
\label{section:finitos}

El estudio $p$-local de grupos finitos no necesariamente abelianos es m\'as sutil
porque en general los grupos finitos no se rompen como productos de $p$-grupos finitos. 

\begin{definition}
\label{subconjugacion}
Sean $P$ y $Q$ subgrupos de $G$. Decimos que un homomorfismo $f \colon P \to Q$ es una $G$--subconjugaci\'on
si existe $g \in G$ tal que $f(x) = gxg^{-1}$ para todo $x \in P$. Si adem\'as $f$ es un isomorfismo, diremos
que es una $G$--conjugaci\'on.
\end{definition}

En cualquiera de los dos casos, denotamos $f=c_g$. Si existe una $G$--conjugaci\'on entre dos subgrupos de $G$,
diremos que son $G$--conjugados. Recordemos que cualquier grupo finito tiene $p$-subgrupos maximales, 
llamados $p$-subgrupos de Sylow, y que cualesquiera dos de estos son $G$--conjugados.

\begin{definition}
Sean $G$, $H$ grupos finitos y $p$ un primo. Se dice que $G$ es $p$-localmente equivalente
a $H$ si existe un isomorfismo $\alpha \colon S \to R$ entre ciertos $p$-subgrupos de Sylow 
$S$, $R$ de $G$ y $H$, respectivamente, que satisface:
\begin{itemize}
\item Para cada par de subgrupos $P$, $Q \leq S$ y cada homomorfismo $f \colon P \to Q$, se tiene que $f$ es
una $G$--subconjugaci\'on si y solo si $\alpha f \alpha^{-1}$ es una $H$--subconjugaci\'on.
\end{itemize}
\end{definition}

A $\alpha$ se le llama un isomorfismo que preserva fusi\'on o una equivalencia $p$-local. La condici\'on de la definici\'on es equivalente 
a pedir que $c_{\alpha}$ env\'ie $G$--subconjugaciones a $H$--subconjugaciones y que $c_{\alpha^{-1}}$ env\'ie $H$--subconjugaciones a $G$--subconjugaciones. 
Si $G$ es $p$-localmente equivalente a $H$, denotamos $G \simeq_p H$. Algunas propiedades de equivalencias $p$-locales que son inmediatas de la definici\'on son las siguientes.
\begin{enumerate}
\item La composici\'on de equivalencias $p$-locales es una equivalencia $p$-local.
\item Si $G$ es $p$-localmente equivalente a $H$ y $G'$ es $p$-localmente equivalente a $H'$, entonces $G \times G'$ es $p$-localmente equivalente
a $H \times H'$.
\item Dos $p$-grupos finitos son $p$-localmente equivalentes si y solo si son isomorfos.
\end{enumerate}
Si $G$ y $H$ son grupos abelianos finitos, las \'unicas subconjugaciones son las inclusiones, as\'i que
cualquier isomorfismo entre sus $p$-subgrupos de Sylow autom\'aticamente preserva fusi\'on. Pero adem\'as, sus $p$-subgrupos de Sylow
son isomorfos a $G \pcom$ y $H \pcom$, respectivamente. Con lo cual estamos extendiendo la Definici\'on \ref{EquivalenciapLocalAbelianos}.

\begin{remark}
Si existe un isomorfismo $\alpha \colon S \to R$ como en la definici\'on anterior y $S'$ es otro $p$-subgrupo de Sylow de $G$, existe
$g \in G$ tal que $S=gS'g^{-1}$ y entonces podemos considerar la composici\'on $ \alpha c_g \colon S' \to R$, que tambi\'en es un isomorfismo.
Dados $P$, $Q \leq S'$ y un homomorfismo $f \colon P \to Q$, se tiene
\[ \alpha c_g f (\alpha c_g)^{-1} = \alpha c_g f c_g^{-1} \alpha^{-1} = \alpha c_g f c_{g^{-1}} \alpha^{-1}. \]
Notemos que $f$ es una $G$--conjugaci\'on si y solo si lo es $c_g f c_{g^{-1}}$. Como $\alpha$ preserva fusi\'on, se tiene que $c_g f c_{g^{-1}}$
es una $G$--conjugaci\'on si y solo si $\alpha c_g f c_{g^{-1}} \alpha^{-1}$ es una $G'$--conjugaci\'on. Hemos probado que $\alpha c_g$ tambi\'en
es un isomorfismo que preserva fusi\'on. Un argumento similar funcionar\'ia si elegimos otro $p$-subgrupo de Sylow de $G'$, as\'i que si existe un
isomorfismo que preserva fusi\'on entre ciertos $p$-subgrupos de Sylow, tambi\'en lo existe para cualesquiera $p$-subgrupos de Sylow.
\end{remark}

\begin{example}
\label{SigmaTres2Nilpotente}
Veamos que $\Z/2 \simeq_2 \Sigma_3$. Tomemos los $2$-subgrupos de Sylow $S=\Z/2$ y $ R = \{ 1, (1,2) \}$ para $\Z/2$ y $\Sigma_3$, respectivamente.
Consideremos el isomorfismo $\alpha \colon S \to R$ que env\'ia $[1]_2$ a $(1,2)$. Sea $f \colon P \to Q$ un homomorfismo entre subgrupos de $S$.
Notemos que $f$ es el homomorfismo trivial si y solo si lo es $ \alpha f \alpha^{-1}$. El homomorfismo trivial es una $\Z/2$--subconjugaci\'on si y solo si
$P \leq Q$, lo cual ocurre si y solo si $\alpha(P) \leq \alpha(Q)$. Solo resta comprobar el caso de $ 1_S \colon S \to S$, que es una $\Z/2$--subconjugaci\'on,
para el cual $\alpha 1_S \alpha^{-1} = 1_R$, que es una $\Sigma_3$--subconjugaci\'on.
\end{example}

\begin{example}
\label{IsomorphicSylows}
Veamos que $ \Z/3 \not\simeq_3 \Sigma_3 $. Tomamos $S=\Z/3$ y $R = \{ 1,(1,2,3),(1,3,2) \}$ como sus $3$-subgrupos de Sylow. Hay dos posibles isomorfismos $S \to R$. El primero es el isomorfismo
$\alpha$ que env\'ia $[1]_3$ a $(1,2,3)$. Sea $f \colon S \to S$ el homomorfismo dado por $f([m]_3) = [2m]_3$, que no es una $\Z/3$--conjugaci\'on. Entonces $\alpha f \alpha^{-1}$
es el \'unico homomorfismo que env\'ia $(1,2,3)$ a $(1,3,2)$, que es la conjugaci\'on por $(1,2)$. As\'i que $\alpha$ no preserva fusi\'on. El otro isomorfismo $\beta$
env\'ia $[1]_3$ a $(1,3,2)$, pero igualmente cumple que $\beta f \beta^{-1}$ es conjugaci\'on por $(1,2)$. 
\end{example}

Estos dos ejemplos fueron f\'aciles de tratar manualmente pues los $p$-subgrupos de Sylow ten\'ian pocos subgrupos y hab\'ia pocos
homomorfismos entre ellos, pero en general este m\'etodo ser\'ia ineficiente. As\'i que usaremos invariantes $p$-locales, que al 
igual que en el caso de grupos abelianos finitos, es una asignaci\'on que env\'ia grupos $p$-localmente equivalentes a objetos
isomorfos. Comencemos listando un par de invariantes $p$-locales sencillos.

\begin{enumerate}
\item La m\'axima potencia de $p$ que divide el orden del grupo.
\item El $p$-subgrupo de Sylow del grupo.
\end{enumerate}

Ya vimos en el Ejemplo \ref{IsomorphicSylows} dos grupos finitos que no son $3$-localmente equivalentes con $3$-subgrupos de Sylow isomorfos. Tiene
sentido que esto pueda pasar pues la definici\'on de equivalencia $p$-local no solo requiere un isomorfismo, sino que tambi\'en preserve fusi\'on. 
Para obtener un mejor invariante, necesitamos considerar c\'omo un grupo finito conjuga los elementos de su $p$-subgrupo de Sylow. Esto se puede apreciar por ejemplo en el conjunto
\[ \cc_p(G) = \{ \text{ clases de conjugaci\'on de elementos de $G$ de orden una potencia de $p$ } \}. \]
Dados elementos $x$, $y$ de un $p$-subgrupo de Sylow $S$ de $G$, diremos que $x \sim_G y$ si son $G$--conjugados. El orden de los elementos de $S$
es una potencia de $p$, as\'i que hay una funci\'on natural
\begin{align*}
S/{\sim_G} & \to \cc_p(G), \\
[s] & \mapsto [s],
\end{align*}
que es claramente inyectiva. Por otra parte, si el orden de $x \in G$ es una potencia de $p$, el subgrupo
generado por $x$ es un $p$-grupo. Por lo tanto, existe $g \in G$ tal que $gxg^{-1} \in S$ y entonces 
$[x]=[gxg^{-1}]$ est\'a en la imagen.

\begin{proposition}
\label{ClasesConjugacion}
Si $G$ y $H$ son $p$-localmente equivalentes, hay una biyecci\'on entre $\cc_p(G)$ y $\cc_p(H)$.
\end{proposition}

\begin{proof}
Escojamos $p$-subgrupos de Sylow $S$ y $R$ de $G$ y $H$, respectivamente. Por los comentarios previos
a la proposici\'on es suficiente encontrar una biyecci\'on entre $S/{\sim_G}$ y $R/{\sim_H}$. Sea 
$\alpha \colon S \to R$ un isomorfismo que preserva fusi\'on y definamos
\begin{align*}
S/{\sim_G} & \to R/{\sim_H}, \\
[x] & \mapsto [\alpha(x)].
\end{align*} 
Veamos que est\'a bien definida. Si $[x]=[y]$, entonces $ y = gxg^{-1}$ para alg\'un $g \in G$.
Sea $P$ el subgrupo generado por $x$ y $Q$ el subgrupo generado por $y$. Entonces $g$ define
una $G$--conjugaci\'on $ c_g \colon P \to Q$ y $\alpha c_g \alpha^{-1} \colon \alpha(P) \to
\alpha(Q)$ debe ser una $H$--subconjugaci\'on, digamos por $h \in H$. Es decir, 
\[ c_h(\alpha(x)) = \alpha c_g \alpha^{-1}(\alpha(x)) = \alpha c_g(x) = \alpha(y), \]
luego $[\alpha(x)] = [\alpha(y)]$. Esta funci\'on es biyectiva pues tiene una inversa 
que env\'ia $[y]$ a $[\alpha^{-1}(y)]$.
\end{proof}

\begin{example}
Veamos que los grupos $A_4$ y $D_{12}$ no son $2$-localmente equivalentes usando la Proposici\'on \ref{ClasesConjugacion}.
Un $2$-subgrupo de Sylow de $A_4$ es
\[ S = \{ 1, (1,2)(3,4), (1,4)(3,2), (1,3)(2,4) \}, \]
mientras que un $2$-subgrupo de Sylow de $D_{12}$ es
\[ R = \{ 1, r^3, s, sr^3 \}. \] 
Ambos subgrupos son isomorfos a $\Z/2 \oplus \Z/2$, as\'i que el tipo de isomorfismo
del $2$-subgrupo de Sylow no sirve en este caso para probar que no son $2$-localmente 
equivalentes. Los elementos no triviales de $S$ son $A_4$--conjugados:
\begin{gather*}
(1,2,3)(1,2)(3,4)(1,2,3)^{-1} = (1,4)(2,3), \\
(1,2,4)(1,2)(3,4)(1,2,4)^{-1} = (1,3)(2,4),
\end{gather*}
as\'i que $S/{\sim_{A_4}}$ tiene dos elementos. Sin embargo, $r^3$ pertenece al centro de $D_{12}$,
luego $R/{\sim_{D_{12}}}$ tiene al menos tres elementos. Concluimos que $A_4 \not\simeq_2 D_{12}$.
\end{example}

Para el siguiente invariante $p$-local, as\'i como para la siguiente secci\'on, necesitamos recordar 
varios conceptos de teor\'ia de grupos y su notaci\'on.

\begin{definition}
Sean $H$ y $K$ subgrupos de $G$.
\begin{enumerate}
\item El normalizador de $H$ en $G$ es
\[ N_G(H) = \{ g \in G \mid gHg^{-1} = H \}. \] 
\item El transportador de $H$ en $K$ es
\[ N_G(H,K) = \{ g \in G \mid gHg^{-1} \subseteq K \}. \]
\item El conjunto de $G$--subconjugaciones de $H$ en $K$ es
\[ \Hom_G(H,K) = \{ f \colon H \to K \mid f \text{ es una $G$--subconjugaci\'on } \}.  \]
\item El grupo de $G$--automorfismos de $H$ es
\[ \Aut_G(H) = \{ f \colon H \to H \mid f \text{ es una $G$--conjugaci\'on } \}.  \]
\item El grupo de automorfismos internos de $G$ es
\[ \Inn(G) = \Aut_G(G). \]
\item El grupo de $G$--automorfismos externos de $H$ es
\[ \Out_G(H) = \Aut_G(H)/\Inn(H). \]
\item El centralizador de $H$ en $G$ es
\[ C_G(H) = \{ g \in G \mid ghg^{-1} = h \text{ para todo } h \in H \}. \]
\item El centro de $G$ es
\[ Z(G) = C_G(G). \]
\end{enumerate}
\end{definition}

Puesto que las equivalencias $p$-locales preservan subconjugaciones, el siguiente resultado no
ser\'a sorprendente.

\begin{proposition}
Si $G \simeq_p H$ y $S$, $R$ son $p$-subgroups de Sylow de $G$ y $H$, respectivamente, entonces
$ \Aut_G(S) \cong \Aut_H(R)$.
\end{proposition}

\begin{proof}
Sea $\alpha \colon S \to R$ un isomorfismo que preserva fusi\'on. Consideremos el homomorfismo
\begin{align*}
\Aut_G(S) & \to \Aut_H(S'), \\
f & \mapsto \alpha f \alpha^{-1}.
\end{align*}
Est\'a bien definido, pues si $f$ es una $G$--conjugaci\'on, entonces $\alpha f \alpha^{-1}$
es una $H$--conjugaci\'on. Es un isomorfismo, pues su inversa es conjugar por $\alpha^{-1}$.
\end{proof}

\begin{example}
El grupo $AGL_1(\F_5)$ de transformaciones afines del campo $\F_5$ tiene como elementos
a funciones $\F_5 \to \F_5$ de la forma $f_{a,b}(x)=ax + b$ con $a$, $b \in \F_5$ y $a \neq 0$.
Veamos que $D_{10} \not\simeq_5 AGL_1(\F_5)$. Un $5$-subgrupo de Sylow de $D_{10}$ es
\[ S = \{ 1,r,r^2,r^3,r^4 \}, \]
donde $r$ y $s$ son los generadores usuales de $D_{10}$ que satisfacen $r^5 =1$, $s^2=1$ y $srs=r^4$.
Por otra parte, un $5$-subgrupo de Sylow de $AGL_1(\F_5)$ est\'a dado por
\[ R = \{ f_{[1]_5,b} \mid b \in \F_5 \}. \] 
Como $S$ tiene \'indice dos en $D_{10}$, es normal y por lo tanto
\[ \Aut_{D_{10}}(S) = \{ 1_S, c_s \} \cong \Z/2. \]
Por otra parte, en $AGL_1(\F_5)$ se tiene
\[ f_{a,b}^{-1} = f_{a^{-1},-a^{-1}b} \]
y por lo tanto
\[ f_{a,b} f_{[1]_5,c} f_{a,b}^{-1} = f_{[1]_5,ac}. \]
Luego $R$ es normal en $AGL_1(\F_5)$ y se cumple 
\[ \Aut_{AGL_1(\F_5)}(R) = \{ f_{a,[0]_5} \mid a \in \F_5^{\times} \} \cong \Z/4. \]
En consecuencia $D_{10} \not\simeq_5 AGL_1(\F_5)$.
\end{example}

En el caso de grupos con $p$-subgrupos de Sylow abelianos, una versi\'on mejorada
de este invariante caracteriza las equivalencias $p$-locales. Para ello primero
necesitaremos un lema.

\begin{lemma}
\label{LemaExtension}
Sea $G$ un grupo finito con $p$-subgrupo de Sylow $S$ abeliano. Cualquier $G$--subconjugaci\'on
entre subgrupos de $S$ extiende a una $G$--conjugaci\'on de $S$.
\end{lemma}

\begin{proof}
Sea $c_g \colon P \to Q$ una $G$--subconjugaci\'on entre subgrupos de $S$. Como $S$ es abeliano, est\'a contenido en $C_G(P)$ y en 
$C_G(gPg^{-1})$. Puesto que los \'ordenes de estos centralizadores dividen al orden de $G$, se debe tener que $S$ es un $p$-subgrupo 
de Sylow de ambos. Por otra parte se tiene $gC_G(P)g^{-1} \leq C_G(g Pg^{-1})$ y entonces $gSg^{-1}$ es un subgrupo de $C_G(gPg^{-1})$
del mismo orden que $S$, as\'i que tambi\'en es un $p$-subgrupo de Sylow de $C_G(gPg^{-1})$. Entonces debe existir $z \in C_G(gPg^{-1})$ 
tal que $zgSg^{-1}z^{-1}=S$. La conjugaci\'on $c_{zg} \colon S \to S$ restringida a $P$ coincide con $c_g$.
\end{proof}

\begin{proposition}
\label{SylowsAbelianos}
Sean $G$ y $H$ grupos finitos con $p$-subgrupos de Sylow $S$ y $R$ que son abelianos.
Para un isomorfismo $\alpha \colon S \to R$ las siguientes condiciones son equivalentes.
\begin{enumerate}
\item El isomorfismo $\alpha$ preserva fusi\'on.
\item Para cada $f \in \Aut(S)$, se tiene que $f$ es una $G$--conjugaci\'on
si y solo si $\alpha f \alpha^{-1}$ es una $H$--conjugaci\'on.
\item El homomorfismo $ c_{\alpha} \colon \Aut(S) \to \Aut(R)$ restringe
a un isomorfismo $ \Aut_G(S) \to \Aut_H(R)$.
\end{enumerate}
\end{proposition} 

\begin{proof}
Es claro que la primera condici\'on implica la segunda, y que la segunda implica la tercera. Supongamos
que $c_{\alpha}$ restringe a un isomorfismo $ \Aut_G(S) \to \Aut_H(R)$. Sea $c_g \colon P \to Q$ una $G$--subconjugaci\'on
entre subgrupos de $S$. Por el Lema \ref{LemaExtension}, existe $z \in N_G(S)$ tal que la restricci\'on de $c_z \colon S \to S$ 
a $P$ coincide con $c_g$. Por hip\'otesis, se tiene $\alpha c_z \alpha^{-1} = c_h$ para alg\'un $h \in H$ y por lo tanto
$\alpha c_g \alpha^{-1} = c_h$. El mismo argumento muestra que $c_{\alpha^{-1}}$ env\'ia $H$--subconjugaciones a $G$--conjugaciones,
as\'i que $\alpha$ preserva fusi\'on.
\end{proof}

Introducimos ahora una nueva propiedad $p$-local que nos ayudar\'a a determinar una clase de grupos
finitos para los cuales ser isomorfos es lo mismo que ser $p$-localmente equivalentes para todo primo
$p$.

\begin{definition}
Se dice que un grupo $G$ es $p$-nilpotente si es $p$-localmente
equivalente a su $p$-subgrupo de Sylow.
\end{definition}

Por ejemplo, vimos que $\Sigma_3$ es $2$-nilpotente en el Ejemplo \ref{SigmaTres2Nilpotente}, y que no es $3$-nilpotente
en el Ejemplo \ref{IsomorphicSylows}.

Sea $G \simeq_p H$. Si $G$ es $p$-nilpotente, entonces $G$ es $p$-localmente
equivalente a su $p$-subgrupo de Sylow $S$ y por lo tanto $H \simeq_p S$. El
isomorfismo entre los $p$-subgrupos de Sylow de $G$ y $H$ nos dice entonces
que $H$ es $p$-nilpotente. As\'i que ser $p$-nilpotente es una propiedad $p$-local.

\begin{proposition}
Sea $G$ un grupo finito y $S$ un $p$-subgrupo de Sylow. Las siguientes condiciones son equivalentes.
\begin{enumerate}
\item Los elementos de orden primo a $p$ forman un subgrupo.
\item Existe un subgrupo normal $N$ de $G$ tal que la composici\'on $S \to G \to G/N$ es un isomorfismo.
\item El grupo $G$ es $p$-nilpotente.
\item Para cada subgrupo $P$ de $S$, el grupo $\Aut_G(P)$ es un $p$-grupo.
\end{enumerate}
\end{proposition}

\begin{proof}
Supongamos que los elementos de orden primo a $p$ forman un subgrupo $N$. Este subgrupo debe
ser normal pues si un elemento tiene orden primo a $p$, tambi\'en se cumple esto para sus conjugados.
Veamos que la composici\'on $S \to G \to G/N$ es biyectiva. Si $s \in S$ es tal que $sN=N$, entonces
$s \in S \cap N$ con lo que su orden es a la vez primo a $p$ y una potencia de $p$. Luego $s=1$ y
la composici\'on es inyectiva. Para ver la sobreyectividad, es suficiente probar que $|S| \geq |G/N|$.
Sea $d=|G|/|S|$. El orden de los elementos de $N$ divide a $d$, y si el orden de un elemento de
$G$ divide a $d$, en particular es primo a $p$ y por lo tanto est\'a en $N$. Es decir, 
\[ N = \{ x \in G \mid x^d = 1 \}. \]
Por el Teorema de Frobenius (Teorema 9.1.2 en \cite{H}), la cardinalidad del conjunto $\{ x \in G \mid x^d = 1 \}$ es un m\'ultiplo de $d$, de
donde obtenemos
\[ \frac{|G|}{|S|} \leq |N| \Rightarrow |G/N| \leq |S|. \]
Luego $S \to G/N$ es un isomorfismo.

Supongamos ahora que existe un subgrupo normal $N$ de $G$ tal que $S \to G/N$ es un isomorfismo. Conjugaci\'on
por el homomorfismo $1_S \colon S \to S$ claramente env\'ia $S$--subconjugaciones a $G$--subconjugaciones. Su
inversa es $1_S$ y al conjugar $c_g \colon P \to Q$ volvemos a obtener $c_g$, que veremos que tambi\'en
es una $S$--subconjugaci\'on. Como $S \to G/N$ es un isomorfismo, existe $s \in S$ y $n \in N$ tal que $g=sn$.
Dado $p \in P$, el elemento $pn^{-1}p^{-1}$ est\'a en $N$, luego $npn^{-1}p^{-1}$ est\'a en $N$. Entonces
\[ snpn^{-1}p^{-1}s^{-1} \in N \cap S \]
y por ser $S \to G/N$ un isomorfismo, se tiene $snpn^{-1}p^{-1}s^{-1} =1$, de donde $snpn^{-1}s^{-1} = sps^{-1}$.
Esto prueba que $c_g=c_s$, con lo cual $1_S$ es una equivalencia $p$-local de $S$ a $G$.     

Si $G$ es $p$-nilpotente, existe una equivalencia $p$-local $\alpha \colon S \to S$ de $S$ a $G$. Puesto que $\alpha^{-1} \colon S \to S$ 
es una equivalencia $p$-local de $S$ a $S$, la composici\'on con $\alpha$, que es $1_S$, es una equivalencia $p$-local de $S$ a $G$. Sea 
$P \leq S$ y $c_g \in \Aut_G(P)$. Como $1_S$ es una equivalencia $p$-local, se debe tener $c_g = c_s$ para alg\'un $s \in S$ y por lo tanto 
el orden de $c_g$ es una potencia de $p$. En consecuencia, $\Aut_G(P)$ es un $p$-grupo.

Si $\Aut_G(P)$ es un $p$-grupo para todo $P \leq S$, por el Teorema del $p$-complemento normal de Frobenius (Teorema 5.26 de \cite{I}), existe un subgrupo
normal $N$ cuyo \'indice es $|S|$. El orden de $N$ es primo a $p$, as\'i que sus elementos tienen orden primo a $p$. Si $x$ es un elemento de orden $q$ primo
a $p$, entonces $(xN)^q = N$. Pero tambi\'en se tiene $(xN)^{|S|} = N$, as\'i que el orden de $xN$ es primo a $p$ y una potencia de $p$ al mismo tiempo, por lo que
$xN=N$, es decir, $x \in N$. Esto muestra que el conjunto de elementos de orden primo a $p$ es $N$, que es un subgrupo.
\end{proof}

\begin{example}
$A_4$ no es $2$-nilpotente, pues los elementos de orden primo a $2$ son
\[ 1, (1,2,3), (1,2,4), (2,3,4), (1,3,4), (1,3,2), (1,4,2), (2,4,3), (1,4,3), \]
pero $A_4$ no puede tener un subgrupo de orden $9$. Los elementos de orden primo a $3$ son
\[ 1, (1,2)(3,4), (1,3)(2,4), (1,4)(2,3), \]
que forman un subgrupo, as\'i que $A_4$ es $3$-nilpotente. Finalmente $A_4$ es $p$-nilpotente si $p \geq 5$ pues
estos primos no dividen el orden de $A_4$.
\end{example}

\begin{example}
El grupo $\Z/4$ act\'ua sobre $\Z/3$ por automorfismos mediante
\[ [n]_4 \cdot [m]_3 = \left\{ \begin{array}{ll}
                               {}[m]_3, & \text{ si $n$ es par, } \\
                               {}[2m]_3, & \text{ si $n$ es impar. } \end{array} \right. \]
Al correspondiente producto semidirecto se le conoce como el grupo dic\'iclico de orden 
$12$ y se denota $\Dic_{12}$. Veamos que $\Dic_{12}$ es $2$-nilpotente.

El orden del elemento $([m]_3,[n]_4)$ tiene que ser un m\'ultiplo del orden de $[n]_4$, que
es $1$, $2$ \'o $4$. Para que este orden sea impar, tiene que ser $1$, es decir, $[n]_4 = [0]_4$. 
Por lo tanto, los \'unicos elementos de orden impar son
\[ ([0]_3,[0]_4), ([1]_3,[0]_4), ([2]_3,[0]_4), \]
que forman un subgrupo. Luego $\Dic_{12}$ es $2$-nilpotente. Su $2$-subgrupo de Sylow es
\[ \{ ([0]_3,x) \mid x \in \Z/4 \} \cong \Z/4, \]
as\'i que $\Dic_{12} \simeq_2 \Z/4$. 
\end{example}

Si $\varphi \colon G \to H$ es un isomorfismo de grupos finitos y $S$ es un $p$-subgrupo de Sylow de $G$,
se tiene que $\varphi(S)$ es un $p$-subgrupo de Sylow de $H$. Adem\'as, la restricci\'on $\varphi \colon S \to \varphi(S)$
es un isomorfismo que preserva fusi\'on ya que si $c_g \colon P \to Q$ es una $G$--subconjugaci\'on entre subgrupos de $S$ se
tiene
\[ \varphi c_g \varphi^{-1} = c_{\varphi(g)}. \]
Igualmente $c_{\varphi^{-1}}$ env\'ia $H$--subconjugaciones a $G$--subconjugaciones. As\'i 
que si dos grupos finitos son isomorfos, son $p$-localmente equivalentes para todo primo $p$. Sin embargo, el converso no es
necesariamente cierto como muestra el siguiente ejemplo.

\begin{example}
\label{ContraejemploEquivalencias}
En este ejemplo veremos que $\Sigma_3 \times \Z/4$ y $\Dic_{12} \times \Z/2$
no son isomorfos, pero son $p$-localmente equivalentes para todo primo $p$. 

Para ver que no son isomorfos, usaremos que si $G \cong H$, entonces $Z(G) \cong Z(H)$. El centro de $\Sigma_3$ es $\{ 1 \}$, as\'i que
\[ Z(\Sigma_3 \times \Z/4) = \{ 1 \} \times \Z/4 \cong \Z/4. \] 
Si $([m]_3,[n]_4) \in Z(\Dic_{12})$, entonces
\[ ([2m]_3,[n+1]_4)= ([0]_3,[1]_4)([m]_3,[n]_4) = ([m]_3,[n]_4)([0]_3,[1]_4) = ([m]_3,[n+1]_4), \]
de donde $m=0$. Luego es de la forma $([0]_3,[n]_4)$. Pero tambi\'en se debe cumplir
\[ ([1]_3,[n]_4)=([1]_3,[0]_4)([0]_3,[n]_4) = ([0]_3,[n]_4)([1]_3,[0]_4) = ([2^n]_3, [n]_4), \]
luego $n=0,2$. Comprobamos que  
\[ ([0]_3,[2]_4)([m]_3,[n]_4) = ([m]_3,[n+2]_4) = ([m]_3,[n]_4)([0]_3,[2]_4), \]
por lo tanto $Z(\Dic_{12}) = \{ ([0]_3,[0]_4), ([0]_3,[2]_4) \} \cong \Z/2$ y entonces
\[ Z(\Dic_{12} \times \Z/2) \cong \Z/2 \times \Z/2. \]
Puesto que sus centros no son isomorfos, los grupos $ \Dic_{12} \times \Z/2$ y $\Sigma_3 \times \Z/4$
no son isomorfos.

Veamos que $ \Sigma_3 \times \Z/4 \simeq_p \Dic_{12} \times \Z/2 $ para todo primo $p$. Puesto que $2$ y $3$
son los \'unicos primos que dividen los \'ordenes de estos grupos, para $p \geq 5$ se tiene
\[ \Sigma_3 \times \Z/4 \simeq_p \{ 1 \} \simeq_p \Dic_{12} \times \Z/2. \]
Ya probamos que $\Sigma_3$ y $\Dic_{12}$ son $2$-nilpotentes. Sus $2$-subgrupos de Sylow son isomorfos a $\Z/2$
y $\Z/4$, respectivamente as\'i que
\[ \Sigma_3 \times \Z/4 \simeq_2 \Z/2 \times \Z/4 \simeq_2 \Z/2 \times \Dic_{12}. \]
Por otra parte, consideremos sus $3$-subgrupos de Sylow
\begin{align*}
S & = \{ ([0]_3,[0]_4), ([1]_3,[0]_4), ([2]_3,[0]_4) \}, \\
S' & = \{ 1, (1,2,3), (1,3,2) \}, 
\end{align*}
que son abelianos. Se puede comprobar que $\Aut_{\Dic_{12}}(S) = \Aut(S)$ y $\Aut_{\Sigma_3}(S') = \Aut(S')$, 
as\'i que cualquier isomorfismo $S \to S'$ preserva fusi\'on. Por lo tanto $ \Dic_{12} \simeq_3 \Sigma_3$
y
\[ \Sigma_3 \times \Z/4 \simeq_3 \Sigma_3 \simeq_3 \Dic_{12} \simeq_3 \Dic_{12} \times \Z/2. \] 
\end{example}

Recordemos que la serie central descendente de $G$ se define de manera 
inductiva mediante $\gamma_1(G) = G$ y $\gamma_{n+1}(G) = [\gamma_n(G),G]$.
Forman una serie
\[ \cdots \vartriangleleft \gamma_{n+1}(G) \vartriangleleft \gamma_n(G) \vartriangleleft \cdots \vartriangleleft \gamma_2(G) \vartriangleleft \gamma_1(G) = G . \]
Se dice que un grupo es nilpotente si su serie central descendente termina en el
grupo trivial en un n\'umero finito de pasos. Es decir, si $\gamma_n(G)=\{1\}$ para alg\'un $n$.

\begin{proposition}
Sea $G$ un grupo finito y para cada primo $p$, sea $S_p$ un $p$-subgrupo
de Sylow de $G$. Las siguientes condiciones son equivalentes.
\begin{itemize}
\item $G$ es nilpotente.
\item $G$ es isomorfo a $\prod_p S_p$.
\item $G$ es $p$-nilpotente para todo primo $p$.
\end{itemize}
\end{proposition}

\begin{proof}
Supongamos que $G$ es isomorfo a $\prod_p S_p$. Puesto que los $p$-grupos finitos son nilpotentes (Teorema 5.33 de \cite{R}), obtenemos que $G$ es nilpotente.
Tambi\'en tendr\'iamos 
\[ G \cong \prod_p S_p \simeq_q S_q \]
para cada primo $q$, luego $G$ es $p$-nilpotente para todo primo $p$.

Supongamos que $G$ es nilpotente. Sea $p$ un primo que divide al orden de $G$. Si $N_G(S_p) \neq G$, entonces por el Teorema 5.38 de \cite{R} se tendr\'ia $N_G(S_p) < N_G(N_G(S_p))$. Notemos que $S_p$
es el \'unico $p$-subgrupo de Sylow de $N_G(S_p)$, pues es normal en $N_G(S_p)$. Por lo tanto, tambi\'en es el \'unico $p$-subgrupo de Sylow de $G$
que est\'a contenido en $N_G(S_p)$. Dado $g \in N_G(N_G(S_p))$, cumple $gN_G(S_p)g^{-1} \leq N_G(S_p)$ y en particular $gS_pg^{-1} \leq N_G(S_p)$.
Pero entonces $gS_pg^{-1}=S_p$, de donde $g \in N_G(S_p)$, lo cual contradice $N_G(S_p) < N_G(N_G(S_p))$. Por lo tanto $N_G(S_p)=G$, es decir, 
$S_p$ es normal en $G$. Si $p$ y $q$ son primos distintos que dividen a $G$ y $s \in S_p$, $t \in S_q$, se tiene que $sts^{-1} \in S_q$, luego
$sts^{-1}t^{-1} \in S_q$. Igualmente se tiene que $ts^{-1}t^{-1} \in S_p$, luego $sts^{-1}t^{-1} \in S_p$. Por lo tanto, el orden de este elemento
es una potencia de $p$ y a la vez primo a $p$, luego $s$ conmuta con $t$. Ahora consideremos la funci\'on
\begin{align*}
\prod_p S_p & \to G, \\
(s_p)_p & \mapsto \prod_p s_p.
\end{align*}
Notemos que es un producto finito pues si $p$ no divide al orden de $G$, se tiene $S_p = \{ 1 \}$. Es un homomorfismo porque los elementos $s_p$ conmutan entre s\'i. Si se tiene
\[ \prod_p s_p = \prod_p s'_p, \]
entonces para cada primo $q$ que divida al orden de $G$ se cumple
\[ s_q^{-1}s'_q = \prod_{p \neq q} s_p'^{-1} s_p. \]
Este es un elemento cuyo orden es una potencia de $q$ y a la vez primo a $q$, luego $s_q = s'_q$. As\'i que este
homomorfismo es inyectivo, y como ambos grupos tienen el mismo orden, es un isomorfismo. 

Supongamos ahora que $G$ es $p$-nilpotente para todo primo $p$. Entonces para cada primo $p$ existe un subgrupo normal $N_p$
tal que la composici\'on de la inclusi\'on $i_p \colon S_p \to G$ con el cociente $\pi_p \colon G \to G/N_p$ es un isomorfismo. Sea
$\alpha_p = \pi_p i_p$ y sea $r_p = \alpha_p^{-1} \pi_p \colon G \to S_p$. Satisface $r_p i_p = 1_{S_p}$, as\'i que es una retracci\'on.
Consideremos la funci\'on
\begin{align*}
r \colon G & \to \prod_p S_p, \\
 g & \mapsto (r_p(g))_p.
\end{align*}
Es un homomorfismo pues cada $r_p$ lo es. Adem\'as, si $s \in S_q$, entonces $r(s)$ es la tupla que tiene $s$ en la coordenada $q$ y $1$ en el
resto de coordenadas. Por lo tanto, $r$ es sobreyectiva y como ambos grupos tienen el mismo orden, es un isomorfismo.
\end{proof}

\begin{proposition}
Dos grupos finitos nilpotentes son isomorfos
si y solo si son $p$-localmente equivalentes para todo primo $p$.
\end{proposition}

\begin{proof}
Supongamos que $G \simeq_p H$ para todo primo $p$ y sean $S_p$ y $R_p$ los $p$-subgrupos de Sylow de $G$ y $H$, respectivamente. 
Como ambos son $p$-nilpotentes para todo $p$, se tiene
\[ S_p \simeq_p G \simeq_p H \simeq_p R_p, \]
de donde $S_p \cong R_p$. Ahora como $G$ y $H$ son nilpotentes, se tiene
\[ G \cong \prod_p S_p \cong \prod_p R_p \cong H, \]
como quer\'iamos probar.
\end{proof}

En el Ejemplo \ref{ContraejemploEquivalencias}, vimos que $\Dic_{12} \times \Z/2 \simeq_3 \Sigma_3 \times \Z/4$
y $\Sigma_3 \times \Z/4 \simeq_3 \Sigma_3$, que no es $3$-nilpotente por el Ejemplo \ref{IsomorphicSylows}, as\'i que $\Dic_{12} \times \Z/2$
y $\Sigma_3 \times \Z/4$ no son $3$-nilpotentes. Por lo tanto, no son nilpotentes y esto permite la situaci\'on del
Ejemplo \ref{ContraejemploEquivalencias}.

\section{Sistemas de fusi\'on y cohomolog\'ia de grupos}
\label{section:fusion}

Sea $S$ un $p$-subgrupo de Sylow del grupo finito $G$. Consideremos
la categor\'ia $\Ff_S(G)$ cuyos objetos son los subgrupos de $S$, 
con morfismos
\[ \Mor_{\Ff_S(G)}(P,Q) = \Hom_G(P,Q). \]
Esta categor\'ia se llama el sistema de fusi\'on de
$G$ relativo a $S$ e intuitivamente es un contenedor que
almacena la informaci\'on $p$-local del grupo $G$. Para
justificar esta \'ultima afirmaci\'on, notemos que si $R$ es
un $p$-subgrupo de Sylow de $H$ y $\alpha \colon S \to R$
es un isomorfismo que preserva fusi\'on, induce un isomorfismo
de categor\'ias
\begin{align*}
\alpha_* \colon \Ff_S(G) & \to \Ff_R(H), \\
 P & \mapsto \alpha(P), \\
 f & \mapsto \alpha f \alpha^{-1}.
\end{align*}
Por otra parte, si $\alpha \colon S \to R$ es un isomorfismo tal
que $\alpha_*$ tiene sentido, se debe cumplir que $\alpha c_g \alpha^{-1}$
es una $H$--subconjugaci\'on para cada $G$--subconjugaci\'on $c_g$. Si
adem\'as $\alpha_*$ es un isomorfismo de categor\'ias, dada una $H$--subconjugaci\'on
$c_h$, debe existir una $G$--subconjugaci\'on $c_g$ tal que $\alpha c_g \alpha^{-1} = c_h$.
As\'i que $\alpha^{-1}c_h \alpha = c_g$, es decir, $\alpha$
es un isomorfismo que preserva fusi\'on. A un isomorfismo de categor\'ias de la forma
$\alpha_*$ le llamaremos un isomorfismo de sistemas de fusi\'on. Podemos resumir lo que 
hemos probado en el siguiente corolario.

\begin{corollary}
Sean $S$ y $R$ $p$-subgrupos de Sylow de $G$ y $H$. Entonces $G \simeq_p H$ si y solo si 
existe un isomorfismo $\Ff_S(G) \to \Ff_R(H)$ de sistemas de fusi\'on.
\end{corollary}

Podemos representar el sistema de fusi\'on mediante un diagrama donde colocamos los subgrupos
de $S$ con orden $p^n$ en la fila $(n+1)$-\'esima contando debe abajo. Usaremos segmentos verticales
u oblicuos para inclusiones de subgrupos de una fila en subgrupos de la fila de arriba. Usamos 
segmentos horizontales para indicar que dos subgrupos son $G$--conjugados, pero solo los necesarios
para que la relaci\'on de $G$--conjugaci\'on quede determinada por transitividad. Adem\'as se indican los $G$--automorfismos 
de un subgrupo en cada clase de $G$--conjugaci\'on.

Es f\'acil comprobar que dos sistemas de fusi\'on isomorfos deben compartir el mismo diagrama, salvo reordenaci\'on y elecci\'on del
subgrupo de cada clase donde indicamos los $G$--automorfismos. Por lo tanto este diagrama o cualquier pieza de informaci\'on que extraigamos
de \'el es un invariante $p$-local. Por ejemplo, si $S$ es un $2$-subgrupo de Sylow de $\Sigma_4$, el diagrama 
de $\Ff_S(\Sigma_4)$ es:
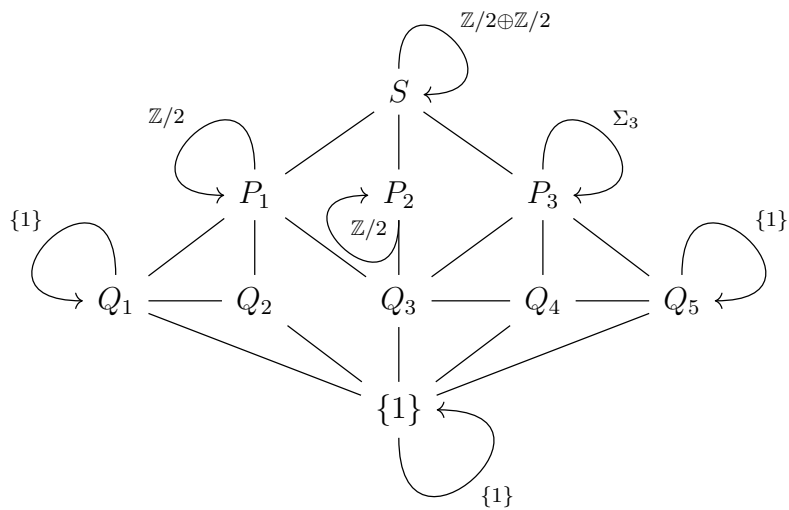
\[
\begin{tikzcd}
& & S \ar[loop, in=0,out=90,looseness=8,"\Z/2 \oplus \Z/2",""'] & & \\
& P_1 \ar[ur,dash] \ar[loop, in=180,out=90,looseness=8,"","\Z/2"'] & P_2 \ar[u,dash] \ar[loop, in=180,out=270,looseness=7,"","\Z/2"'] & P_3 \ar[loop, in=0,out=90,looseness=8,"\Sigma_3", ""'] \ar[ul,dash] & \\
Q_1 \ar[loop, in=180,out=90,looseness=8,"","\{1\}"'] \ar[r,dash] \ar[ur,dash] & Q_2 \ar[u,dash] & Q_3 \ar[r,dash] \ar[u,dash] \ar[ur,dash] \ar[ul,dash] & Q_4 \ar[u,dash] \ar[r,dash] & Q_5 \ar[loop, in=0,out=90,looseness=8,"\{1\}",""'] \ar[ul,dash] \\
 & & \{ 1 \} \ar[loop, in=0,out=270,looseness=8,"","\{1\}"'] \ar[ull,dash] \ar[ul,dash] \ar[u,dash] \ar[ur,dash] \ar[urr,dash] & & 
\end{tikzcd}
\]
Mientras que el diagrama de $\Ff_S(S)$ ser\'ia:
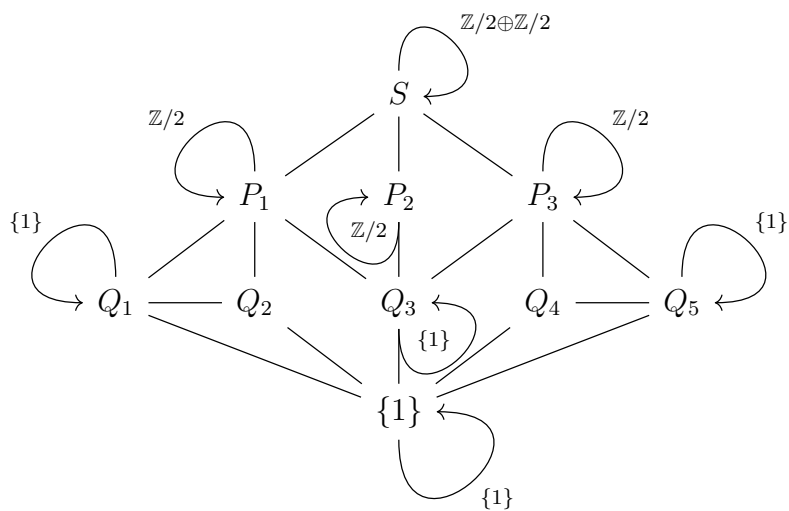
\[
\begin{tikzcd}
& & S \ar[loop, in=0,out=90,looseness=8,"\Z/2 \oplus \Z/2",""'] & & \\
& P_1 \ar[ur,dash] \ar[loop, in=180,out=90,looseness=8,"","\Z/2"'] & P_2 \ar[u,dash] \ar[loop, in=180,out=270,looseness=7,"","\Z/2"'] & P_3 \ar[loop, in=0,out=90,looseness=8,"\Z/2", ""'] \ar[ul,dash] & \\
Q_1 \ar[loop, in=180,out=90,looseness=8,"","\{1\}"'] \ar[r,dash] \ar[ur,dash] & Q_2 \ar[u,dash] & Q_3 \ar[loop, in=0,out=270,looseness=7,"\{1\}",""'] \ar[u,dash] \ar[ur,dash] \ar[ul,dash] & Q_4 \ar[u,dash] \ar[r,dash] & Q_5 \ar[loop, in=0,out=90,looseness=8,"\{1\}",""'] \ar[ul,dash] \\
 & & \{ 1 \} \ar[loop, in=0,out=270,looseness=8,"","\{1\}"'] \ar[ull,dash] \ar[ul,dash] \ar[u,dash] \ar[ur,dash] \ar[urr,dash] & & 
\end{tikzcd}
\]
La informaci\'on contenida en estos diagramas se puede obtener usando el software GAP \cite{GAP4}. Podemos dar varias razones por las que $S \not\simeq_2 \Sigma_4$.
\begin{enumerate}
\item Tienen un n\'umero distinto de clases de conjugaci\'on de subgrupos.
\item Tienen un n\'umero distinto de clases de conjugaci\'on de subgrupos de orden dos.
\item Hay una clase de $\Sigma_4$--conjugaci\'on con tres subgrupos, pero no hay tal
clase de $S$--conjugaci\'on.
\item Hay una clase de $S$--conjugaci\'on con un subgrupo, pero no existe tal clase 
de $\Sigma_4$--conjugaci\'on.
\item Existe un subgrupo cuyo grupo de automorfismos en $\Ff_S(\Sigma_4)$ es $\Sigma_3$,
pero no existe tal subgrupo para $\Ff_S(S)$.
\end{enumerate}

\begin{remark}
En realidad, el diagrama contiene toda la informaci\'on del sistema de fusi\'on. Para empezar, si dos subgrupos son $G$--conjugados, tienen grupos
de $G$--automorfismos isomorfos mediante el isomorfismo
\begin{align*}
\Aut_G(H) & \to \Aut_G(gHg^{-1}), \\
 c_z & \mapsto c_{gzg^{-1}}.
\end{align*}
Es por ello que solo indicamos el grupo de $G$--automorfismos de un subgrupo en cada clase de $G$--conjugaci\'on. Por otra parte, cualquier 
$G$--subconjugaci\'on $c_g \colon P \to Q$ se puede descomponer como una $G$--conjugaci\'on $ c_g \colon P \to gPg^{-1}$ seguida de una inclusi\'on 
$gPg^{-1} \to Q$.  Por esta raz\'on solamente conectamos las filas entre ellas mediante inclusiones. Por \'ultimo, si existe una $G$--conjugaci\'on
entre $P$ y $Q$, todas las dem\'as $G$--conjugaciones de $P$ a $Q$ son composiciones de esta $G$--conjugaci\'on fija con los $G$--automorfismos de $P$ y por eso solo indicamos
cuando dos subgrupos son $G$--conjugados, pero no todas las $G$--conjugaciones entre ellos.
\end{remark}

Determinar el diagrama de un sistema de fusi\'on requiere bastante trabajo. Es por ello
que conviene manejar ciertas subcategor\'ias que generan al sistema de fusi\'on en cierto
sentido.

\begin{definition}
Sea $K$ un subgrupo de $H$. Se dice que $K$ est\'a fuertemente $p$-encajado
en $H$ si contiene un $p$-subgrupo de Sylow de $H$ y $ | K \cap hKh^{-1} |$ es primo a 
$p$ para todo $h \in H-K$.
\end{definition}

\begin{definition}
Sea $S$ un $p$-subgrupo de Sylow de $G$. Se dice que $P \leq S$ es
\begin{itemize}
\item $G$--c\'entrico si $C_S(gPg^{-1}) = Z(gPg^{-1})$ para todo $g \in G$ tal que $gPg^{-1} \leq S$.
\item $G$--esencial si es $G$--c\'entrico y $\Out_G(P)$ tiene un subgrupo fuertemente $p$-encajado.
\item Completamente $G$--normalizado si $|N_S(P)| \geq |N_S(gPg^{-1})|$ para todo $g \in G$ tal que $gPg^{-1} \leq S$.
\end{itemize}
\end{definition}

El siguiente teorema (Teorema 2.13 en \cite{He}) explica nuestro inter\'es en estos tipos de subgrupos.

\begin{teoremaAlperin}
Sea $S$ un $p$-subgrupo de Sylow de $G$ y $P \leq S$ tal que $ gPg^{-1} \leq S$. Entonces existen
subgrupos $ Q_0, \ldots, Q_{n+1}$ y $ R_1, \ldots, R_{n+1}$, y automorfismos
$ \varphi_j \in \Aut_G(R_j)$ tales que:
\begin{itemize}
\item Cada $R_i$ es completamente $G$--normalizado y $G$--esencial, o es $S$.
\item $P = Q_0$ y $ gPg^{-1} = Q_{n+1}$.
\item $ Q_{i-1}$, $Q_i \leq R_i$ para todo $i$.
\item $ \varphi_j(Q_{j-1}) = Q_j$ para todo $j$.
\item $ c_g(x) = \varphi_{n+1} \circ \cdots \circ \varphi_1(x)$ para todo $x \in P$.
\end{itemize}
\end{teoremaAlperin}

El siguiente diagrama es una representaci\'on gr\'afica del teorema.
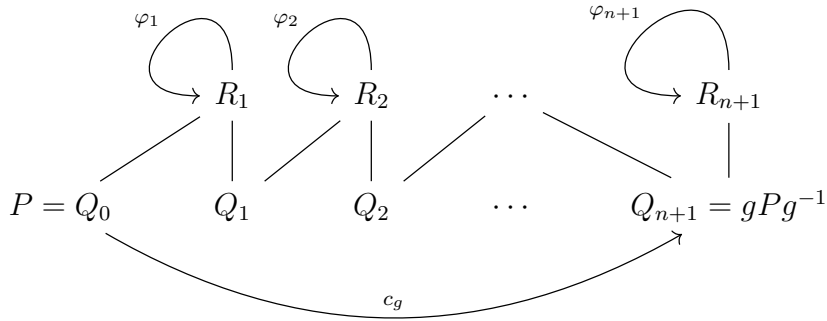
\[
\begin{tikzcd}
        & R_1 \ar[loop, in=180,out=90,looseness=8,"","\varphi_1"'] & R_2 \ar[loop, in=180,out=90,looseness=8,"","\varphi_2"'] & \cdots & R_{n+1} \ar[loop, in=180,out=90,looseness=7,"","\varphi_{n+1}"']             \\
P = Q_0 \ar[ur,dash] \arrow[rrrr,bend right, "c_g"] & Q_1 \ar[u,dash] \ar[ur,dash] & Q_2 \ar[u,dash] \ar[ur,dash] & \cdots & Q_{n+1} = gPg^{-1} \ar[u,dash] \ar[ul,dash]
\end{tikzcd}
\] 

\begin{corollary}
Sean $S$ y $R$ $p$-subgrupos de Sylow de $G$, $H$, respectivamente. Un isomorfismo $\alpha \colon S \to R$
es una equivalencia $p$-local si y solo si preserva subgrupos esenciales y $c_{\alpha} \colon \Aut_G(P)
\to \Aut_H(\alpha(P))$ es un isomorfismo para todo subgrupo $G$--esencial $P$ y para $P=S$.
\end{corollary}

Preservar subgrupos esenciales quiere decir que $P$ es $G$--esencial si y solo si $\alpha(P)$ es $H$--esencial.
Con esto reducimos el problema a estudiar el diagrama de la subcategor\'ia completa $\Ff^e_S(G)$ de $\Ff_S(G)$ 
cuyos objetos son los subgrupos $G$--esenciales y $S$.

\begin{example}
Sea $S$ un $2$-subgrupo de Sylow de $\Sigma_4$. Podemos usar \cite{GAP4} para determinar los diagramas de $\Ff_S^e(\Sigma_4)$
y $\Ff_S^e(S)$.
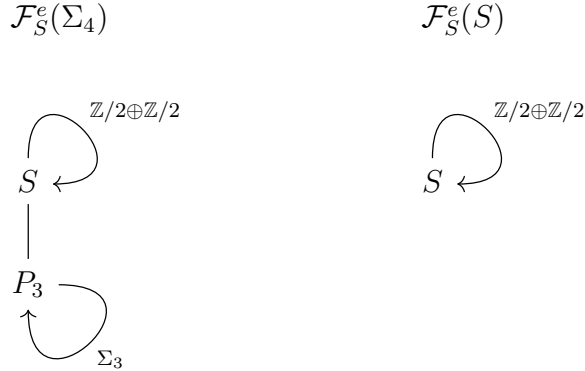
\[
\begin{tikzcd}
\qquad \Ff_S^e(\Sigma_4) & & & \qquad \Ff_S^e(S) \\
 & & \\
S \ar[loop, in=0,out=90,looseness=8,"\Z/2 \oplus \Z/2",""'] &  &  & S \ar[loop, in=0,out=90,looseness=8,"\Z/2 \oplus \Z/2",""'] \\
P_3 \ar[loop, in=270,out=0,looseness=8,"\Sigma_3",""'] \ar[u,dash] & & &   
\end{tikzcd}
\]
\end{example}

Si los $p$-subgrupos de Sylow de $G$ son abelianos, entonces no hay
subgrupos esenciales. Esto proporciona una explicaci\'on m\'as conceptual
de la Proposici\'on \ref{SylowsAbelianos}. El sistema de fusi\'on tambi\'en nos puede ayudar a identificar invariantes
$p$-locales. Si una asignaci\'on se puede expresar en t\'erminos del sistema
de fusi\'on, debe ser un invariante $p$-local. Ejemplos no inmediatamente
evidentes de esto son la homolog\'ia y cohomolog\'ia de grupos con coeficientes
en $\F_p$, el campo con $p$ elementos, definidos mediante
\begin{align*}
H_n(G;\F_p) & = \Tor_n^{\Z G}(\Z,\F_p), \\
H^n(G;\F_p) & = \Ext_n^{\Z G}(\Z,\F_p). 
\end{align*}
Aqu\'i $\Z G$ denota al anillo de grupo de $G$, que es el grupo abeliano libre
con base $G$, con el producto inducido por la multiplicaci\'on de $G$. La
notaci\'on $\Z$ y $\F_p$ se refiere a estos grupos abelianos con su estructura
de $\Z G$--m\'odulos dada por
\[ \left( \sum n_g g \right) \cdot x = \sum n_g x, \]
donde en el lado derecho usamos la estructura de grupo abeliano de $\Z$ y $\F_p$.
Aunque nos referiremos a ellos como los grupos de homolog\'ia y cohomolog\'ia con
coeficientes en $\F_p$, tienen estructura de espacio vectorial sobre $\F_p$. Tambi\'en podemos considerar
\[ H^*(G;\F_p) = \bigoplus_{n \geq 0} H^n(G;\F_p), \]
que no solo es un $\F_p$--espacio vectorial, sino una $\F_p$--\'algebra graduada. 

\begin{teoremaCartan}
Si $G$ es un grupo finito y $S$ es un $p$-subgrupo de Sylow de $G$, entonces
\[ H_n(G;\F_p) \cong \highercolim{\Ff_S(G)}{} H_n(Q;\F_p), \]
\[ H^n(G;\F_p) \cong \higherlim{\Ff_S(G)}{} H^n(Q;\F_p), \]
para cualquier $n \geq 0$. El isomorfismo en cohomolog\'ia tambi\'en induce un isomorfismo
de $\F_p$--\'algebras graduadas
\[ H^*(G;\F_p) \cong \higherlim{\Ff_S(G)}{} H^*(Q;\F_p), \]
\end{teoremaCartan}

La demostraci\'on para cohomolog\'ia se puede encontrar en el Teorema III.10.3 de \cite{Br} y
es an\'aloga para homolog\'ia.

\begin{corollary}
Los grupos de homolog\'ia y cohomolog\'ia de grupos con coeficientes en $\F_p$, as\'i
el anillo de cohomolog\'ia con coeficientes en $\F_p$, son invariantes $p$-locales.  
\end{corollary}

En dimensiones bajas (ver por ejemplo \cite{Br}) obtenemos invariantes $p$-locales con descripciones alternativas.
\begin{align*}
H_1(G;\F_p) & \cong G_{\ab} \otimes_{\Z} \Z/p, \\
H^1(G;\F_p) & \cong \Hom(G,\Z/p), \\
H^2(G;\F_p) & \cong \{ \text{ clases de equivalencia de extensiones centrales de $G$ por $\Z/p$ } \} .
\end{align*}
Es conocido que los grupos de homolog\'ia y cohomolog\'ia de grupos tienen una interpretaci\'on topol\'ogica
en t\'erminos de espacios clasificantes de grupos. En la siguiente secci\'on comentaremos que no solo esto es
cierto, sino que cualquier invariante $p$-local tiene tal interpretaci\'on.

\section{Espacios clasificantes y $p$-completaciones}
\label{section:clasificante}

Para el estudio de acciones libres de un grupo $G$ con un comportamiento local agradable (que tienen rebanadas), es conveniente
usar el lenguaje de haces $G$--principales. La clasificaci\'on de haces $G$--principales est\'a dada en t\'erminos
de clases de homotop\'ia de funciones al espacio clasificante $BG$ de $G$. Estrictamente hablando el espacio clasificante
no es un espacio, sino un tipo de homotop\'ia, pero en \cite{Mil} se construye un funtor de la categor\'ia
de grupos topol\'ogicos a la categor\'ia de espacios que a cada grupo $G$ le asigna un espacio de este tipo de homotop\'ia.
A este espacio se le deber\'ia llamar un modelo para el espacio clasificante de $G$, pero por simplicidad solamente le
llamaremos el espacio clasificante de $G$. Tambi\'en se le conoce como la construcci\'on de Milnor. 

Describiremos ahora una simplificaci\'on de la construcci\'on de Milnor del espacio clasificante que funciona para grupos
finitos $G$, que es nuestro caso de inter\'es en este documento. Dados dos espacios topol\'ogicos $X$ y $Y$, el producto join 
de $X$ y $Y$ es el espacio
\[ X \ast Y = X \times Y \times [0,1] /{\sim}, \]
donde $(x_0,y,0) \sim (x_0,y',0)$ para todo $y$, $y' \in Y$ y todo $x_0 \in X$ 
y $(x,y_0,1) \sim (x',y_0,1) $ para todo $ x$, $x' \in X$ y todo $y_0 \in Y$.
Lo dotamos de la topolog\'ia cociente. Usamos la notaci\'on
\[ tx+(1-t)y := [(x,y,1-t)] \] 
para pensar que $X \ast Y$ es el espacio de segmentos abstractos entre los puntos de $X$ y $Y$. Para construir el espacio 
clasificante del grupo finito $G$, realizamos el siguiente proceso inductivo.
\begin{align*}
E^1 G & = G. \\ 
E^{n+1} G & = E^n G \ast G.   
\end{align*}
Existen funciones continuas
\begin{align*}
E^n G & \to E^{n+1} G, \\
t_1 g_1 + \ldots + t_n g_n & \mapsto t_1 g_1 + \ldots + t_n g_n + 0 \cdot 1,
\end{align*}
y podemos considerar el espacio
\[ EG = \highercolim{n \geq 1}{} E^n G. \] 
Las funciones $E^n G \to E^{n+1}G$ son encajes, as\'i que $EG$ es la uni\'on de los $E^n G$
con la topolog\'ia final, es decir, un subconjunto de $EG$ es abierto si y solo si la intersecci\'on
con cada $E^n G$ es abierto en $E^n G$. Los elementos del espacio $EG$ son por lo tanto sumas formales finitas
\[ \sum \lambda_i g_i, \]
donde $\lambda_i \in [0,1]$ y $\sum \lambda_i = 1$, es decir, combinaciones lineales convexas formales. El grupo $G$ 
act\'ua sobre $EG$ por la derecha mediante
\[ \left( \sum \lambda_i g_i \right) \cdot g = \sum \lambda_i (g_i g), \]
y es f\'acil comprobar que es una acci\'on libre. El espacio clasificante de $G$ 
es el espacio de \'orbitas
\[ BG = EG/G \]
con la topolog\'ia cociente. Aunque este modelo de Milnor es expl\'icito, no siempre es el m\'as conveniente.
Siguiendo los comentarios en el primer p\'arrafo de esta secci\'on, igualmente llamaremos espacio clasificante 
a cualquier espacio que sea homot\'opicamente equivalente a este. El siguiente resultado (Teorema II.1.2 en \cite{AM}) 
nos permite identificar espacios clasificantes.

\begin{proposition}
Sea $E$ un CW-complejo contr\'actil con una acci\'on libre de un grupo finito $G$. Entonces $E/G$ es
un modelo para $BG$. 
\end{proposition}

\begin{example}
El grupo $\Z/2$ act\'ua libremente sobre $S^{\infty}$ mediante la acci\'on antipodal, de
donde $\R P^{\infty}$ es un modelo para $B\Z/2$.
\end{example}

Aparte de su utilidad para clasificar haces $G$--principales, el espacio $BG$ est\'a \'intimamente relacionado
con $G$ por la heur\'istica de que las propiedades e invariantes homot\'opicos de $BG$ deben corresponder
a propiedades e invariantes algebraicos de $G$. Sin intentar ser exhaustivos, he aqu\'i algunas relaciones
que ilustran esta filosof\'ia.
\begin{enumerate}
\item $\pi_1(BG) \cong G$.
\item $H_1(BG) \cong G_{\ab}$.
\item $H_n(BG;M) \cong H_n(G;M)$ y $H^n(BG;M) \cong H^n(G;M)$.
\item $BG \simeq BH$ si y solo si $G \cong H$.
\item $K(BG) \cong R(G)^{\wedge}_{I_G}$.
\end{enumerate}
Estos resultados se pueden encontrar en \cite{AM}, \cite{Br} y \cite{AS}. Tenemos pues un tipo de homotop\'ia que captura la informaci\'on de 
$G$. Es natural preguntarse ahora si existe un tipo de homotop\'ia que contenga la informaci\'on $p$-local de $G$, y esto se lograr\'a siguiendo un proceso
de localizaci\'on en topolog\'ia algebraica. 

Dado un espacio $X$, en \cite{BK} se construye para cada primo $p$ el espacio $X \pcom$, conocido como su $p$-completaci\'on,
y un espacio $X_{\Q}$ llamado su racionalizaci\'on. La construcci\'on de estos espacios no es relevante para esta discusi\'on, sino
m\'as bien lo que se pretende lograr con ellos:
\begin{itemize}
\item Son construcciones funtoriales.
\item El espacio $X \pcom$ captura la informaci\'on $p$-local de $X$, es decir, la informaci\'on
que se puede recuperar a partir de $H_*(X;\F_p)$. M\'as espec\'ificamente, una funci\'on $X \to Y$
induce un isomorfismo en los grupos de homolog\'ia con coeficientes en $\F_p$ si y solo si la funci\'on
inducida $X \pcom \to Y \pcom$ es una equivalencia homot\'opica.
\item Bajo ciertas condiciones, uno puede recuperar el tipo de homotop\'ia de un espacio
a partir de sus $p$-completaciones y su racionalizaci\'on. 
\end{itemize}
Este \'ultimo punto nos permite elaborar estrategias de local a global. Para resolver
un problema, estudiamos primero el problema correspondiente para las completaciones y sus 
racionalizaciones, y despu\'es intentamos recomponer una soluci\'on al problema original
a partir de sus soluciones locales.

\begin{example}
Si $G$ es un grupo finito, el espacio $BG_{\Q}$ es contr\'actil y tambi\'en
lo es $BG \pcom$ si $p$ es un primo que no divide al orden de $G$. Usando
la Proposici\'on VII.4.1 de \cite{BK}, podemos recuperar con el resto
de las $p$-completaciones el tipo de homotop\'ia de
\[ \Z_{\infty} BG \simeq \prod_{p \text{ primo }} BG \pcom, \]
donde $\Z_{\infty} X$ es la $\Z$--completaci\'on de Bousfield-Kan. El espacio
$\Z_{\infty} BG$ tiene los mismos grupos de homolog\'ia que $BG$, pero no 
es homot\'opicamente equivalente a $BG$ en general. Curiosamente, $BG \simeq \Z_{\infty} BG$ si y
solo si $G$ es nilpotente por la Proposici\'on V.3.3 de \cite{BK}.
\end{example}

El siguiente resultado fue enunciado originalmente en \cite{MP}. La demostraci\'on original presentaba 
dificultades y durante un tiempo se conoci\'o como la conjetura de Martino-Priddy hasta que finalmente 
se prob\'o en \cite{O1} y \cite{O2}, usando la teor\'ia de grupos $p$-locales finitos. Diremos m\'as al respecto en la \'ultima secci\'on.

\begin{theorem}
\label{MartinoPriddy}
Sean $G$ y $H$ grupos finitos. Entonces $BG \pcom \simeq BH \pcom$
si y solo si $ G \simeq_p H$.
\end{theorem}
 
Este teorema nos dice que el tipo de homotop\'ia de $BG \pcom$ es un invariante $p$-local completo para grupos finitos. Pero
no es un invariante pr\'actico, pues el problema de determinar cu\'ando dos espacios son homot\'opicamente equivalentes es 
generalmente complicado, y la construcci\'on de $BG \pcom$ a partir de $BG$ tambi\'en lo es.

Sin embargo, la manera en que intentamos entender $BG \pcom$ es a trav\'es de sus invariantes homot\'opicos y
esto tiene valor para nuestro problema ya que un invariante homot\'opico de $BG \pcom$ tambi\'en ser\'a un invariante
$p$-local de $G$. La nueva heur\'istica nos dice que los invariantes y propiedades homot\'opicas de $BG \pcom$ corresponden
a invariantes y propiedades $p$-locales de $G$. Ofrecemos algunos ejemplos de tales invariantes con sus correspondientes
invariantes $p$-locales.
\begin{enumerate}
\item $ H_k(G;\F_p) \cong H_k(BG \pcom;\F_p)$.
\item El $p$-rango de $G$ es igual a la dimensi\'on de Krull de $H^*(BG \pcom;\F_p)$.
\item $G/{O^p(G)} \cong \pi_1(BG \pcom) $.
\item $(G_{\ab}) \pcom \cong H_1(BG \pcom;\Z \pcom) $.
\item $G$ es $p$-nilpotente si y solo si $BS \simeq BG \pcom$.
\item $\Rep(P,G)$ est\'a en biyecci\'on con $[BP,BG \pcom]$.
\end{enumerate}
Estos resultados se pueden consultar en \cite{AM}, \cite{Br}, \cite{BK}, \cite{BLO1}, \cite{CSV} y \cite{Q}.
Aparte de poder obtener nuevos invariantes $p$-locales de esta manera, tambi\'en nos sirve para
tener una raz\'on conceptual de por qu\'e algo es un invariante $p$-local. Por ejemplo, $O^p(G)$
es el subgrupo normal m\'as peque\~no de $G$ cuyo \'indice es una potencia de $p$. No es evidente 
que $G/{O^p(G)}$ solo dependa de $\Ff_S(G)$, pero su descripci\'on como $\pi_1(BG \pcom)$ nos dice 
que as\'i es.

\section{Grupos $p$-locales finitos}
\label{section:grupoplocal}

La teor\'ia de grupos $p$-locales finitos tiene sus antecedentes en las categor\'ias de Frobenius introducidas 
en notas no publicadas en 1993, y que se publicar\'ian en \cite{P}, dirigidas al estudio de las representaciones 
modulares de grupos finitos. Posteriormente se encontr\'o que estas categor\'ias eran valiosas para otros prop\'ositos en topolog\'ia algebraica:
\begin{itemize}
\item Encontrar y analizar espacios con propiedades homot\'opicas similares a $BG \pcom$.
\item Demostrar la conjetura de Martino-Priddy.
\item Determinar descomposiciones (co)homol\'ogicas de $BG$.
\end{itemize}
Merece la pena comentar un poco m\'as sobre el primer objetivo. El grupo de Conway $\Co_3$ es el \'unico grupo simple con su $2$-subgrupo de Sylow. Para demostrar
esto, era suficiente descartar un hipot\'etico grupo simple $G$ distinto de $\Co_3$ con este $2$-subgrupo 
de Sylow $S$ y una cierta condici\'on sobre el centralizador de un elemento de orden dos. En \cite{S} se determin\'o
c\'omo deb\'ia ser $\Ff_S(G)$ si $G$ existiera y se demostr\'o que esto era incompatible con $\Ff_T(G)$, donde $T$
ser\'ia un $p$-subgrupo de Sylow de $G$ para un primo $p$ relacionado con la condici\'on sobre el centralizador. Lo interesante 
de esta demostraci\'on es que no se encontr\'o ninguna inconsistencia en $\Ff_S(G)$, lo cual llev\'o a plantear \cite{Be}
si podr\'ian existir abstracciones de los sistemas de fusi\'on $\Ff_S(G)$ que no proviniesen de un grupo finito, 
pero que estuviera ligado a un espacio topol\'ogico, al igual que $\Ff_S(G)$ lo est\'a con $BG \pcom$. 

Esta abstracci\'on es un sistema de fusi\'on saturado, introducida en \cite{BLO1}. Al igual que $\Ff_S(G)$,
es una categor\'ia de subgrupos de un $p$-grupo finito $S$ y sus morfismos son ciertos monomorfismos. Pero
a diferencia de $\Ff_S(G)$, estos morfismos no necesariamente est\'an dados por subconjugaciones por elementos
de un grupo finito $G$ cuyo $p$-subgrupo de Sylow sea $S$, sino que solo se les requiere que cumplan ciertos
axiomas. Existen sistemas de fusi\'on saturados que no son isomorfos a sistemas de fusi\'on de la forma $\Ff_S(G)$
y a estos se les conoce como ex\'oticos. Las construcciones de \cite{S} y \cite{Be} dieron lugar a tales ejemplos \cite{Lo}.
Otros ejemplos fueron construidos en \cite{RV} sobre $7^{1+2}_+$, el grupo extraespecial de orden $7^3$ y exponente $7$.
Para uno de estos, conocido ahora como $\RV_3$, el diagrama de $\RV_3^e$ est\'a dado por
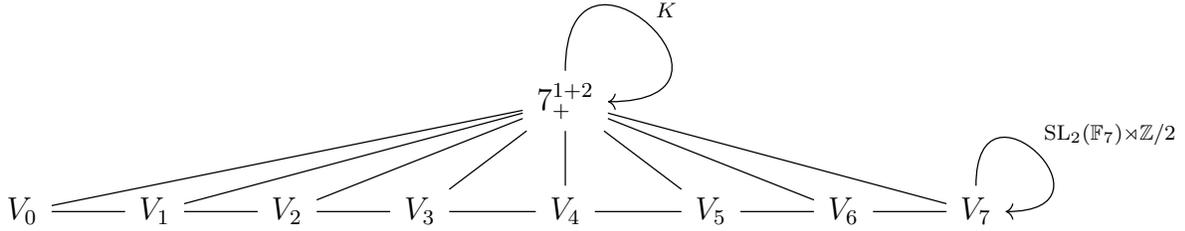
\[ 
\begin{tikzcd}
 & & & & 7^{1+2}_+ \ar[loop, in=0,out=90,looseness=8,"K",""'] & & & \\
V_0 \ar[r,dash] \ar[urrrr,dash] & V_1 \ar[r,dash] \ar[urrr,dash] & V_2 \ar[r,dash] \ar[urr,dash] & V_3 \ar[r,dash] \ar[ur,dash] & V_4 \ar[r,dash] \ar[u,dash] & V_5 \ar[r,dash] \ar[ul,dash] & V_6 \ar[r,dash] \ar[ull,dash] &  V_7 \ar[loop, in=0,out=90,looseness=8,"\SL_2(\F_7) \rtimes \Z/2",""'] \ar[ulll,dash]  
\end{tikzcd}
\]
donde $K$ es una extensi\'on de $\SD_{32} \times \Z/3$ por $7^{1+2}_+$, y los grupos $V_j$ son isomorfos a $\Z/7 \oplus \Z/7$.

A fin de asociarle un espacio topol\'ogico a un sistema de fusi\'on saturado, en \cite{BLO1} se introduce la noci\'on
de grupo $p$-local finito, dado por una terna $\G = (S,\Ff,\Ll)$, donde $S$ es un $p$-grupo finito, $\Ff$ es un sistema
de fusi\'on saturado sobre $S$ y $\Ll$ es un sistema c\'entrico de enlace para $\Ff$. Este \'ultimo es otra categor\'ia
y el espacio clasificante de $\G$ se define como
\[ B\G = |\Ll| \pcom. \]
Por ejemplo, dado un grupo finito $G$ con $p$-subgrupo de Sylow $S$, consideremos la categor\'ia $\Ll_S(G)$ cuyos objetos
son los subgrupos $G$--c\'entricos de $S$, con morfismos
\[ \Mor_{\Ll_S(G)}(P,Q) = N_G(P,Q)/O^p(C_G(P)). \]
Esta categor\'ia es un sistema c\'entrico de enlace para $\Ff_S(G)$ y satisface
\[ |\Ll_S(G)| \pcom \simeq BG \pcom. \]
Esto es consistente con la relaci\'on que establecimos previamente entre $\Ff_S(G)$ y $BG \pcom$, pero parecer\'ia cuestionar
la validez de la conjetura de Martino-Priddy, pues nos dice que es el sistema c\'entrico de enlace quien codifica el tipo de
homotop\'ia de $BG \pcom$. Sin embargo, esto es parte de la estrategia ideada para atacar esta conjetura, que consisti\'o
en demostrar que cada sistema de fusi\'on saturado admite un \'unico sistema c\'entrico de enlace, salvo isomorfismo. Es
decir, el esquema de la demostraci\'on es el siguiente:
\[ G \simeq_p H \Rightarrow \Ff_S(G) \cong \Ff_T(H) \Rightarrow \Ll_S(G) \cong \Ll_T(H) \Rightarrow BG \pcom \simeq BH \pcom . \]
La existencia y unicidad del sistema c\'entrico de enlace asociado a $\Ff_S(G)$ fue probada en \cite{O1} y \cite{O2}, posteriormente extendida
a sistemas de fusi\'on saturados en \cite{C} y re-demostrada en \cite{GL} sin usar la clasificaci\'on de grupos finitos simples.

Aunque ya queda fuera del alcance de este art\'iculo, nos gustar\'ia mencionar brevemente que tambi\'en existen dos maneras de 
tener una visi\'on local de los grupos compactos de Lie. La primera es mediante la teor\'ia de grupos $p$-compactos, que extiende
las ideas de toro maximal, grupo de Weyl y sistemas de ra\'ices para espacios de lazos con condiciones de finitud en el primo $p$.
La segunda es la teor\'ia de grupos $p$-locales compactos, similar en forma a la de grupos $p$-locales finitos. En este caso los
$p$-subgrupos de Sylow son grupos $p$-torales discretos y los conceptos de sistema de fusi\'on saturado y sistema c\'entrico
de enlace se extienden adecuadamente. El lector interesado puede consultar \cite{Grodal} para una vista panor\'amica de los grupos
$p$-compactos y \cite{DW2} para el art\'iculo fundacional de la teor\'ia. Para la teor\'ia de grupos $p$-locales compactos, se puede 
consultar \cite{BLO2}. 

\bibliographystyle{amsplain}

\bibliography{mybibfile}

\providecommand{\bysame}{\leavevmode\hbox to3em{\hrulefill}\thinspace}
\providecommand{\MR}{\relax\ifhmode\unskip\space\fi MR }
\providecommand{\MRhref}[2]{%
  \href{http://www.ams.org/mathscinet-getitem?mr=#1}{#2}
}
\providecommand{\href}[2]{#2}
\begin{thebibliography}{10}

\bibitem{AM}
Alejandro Adem and R.~James Milgram, \emph{Cohomology of finite groups}, second
  ed., Grundlehren der mathematischen Wissenschaften [Fundamental Principles of
  Mathematical Sciences], vol. 309, Springer-Verlag, Berlin, 2004. \MR{2035696}

\bibitem{AKO}
Michael Aschbacher, Radha Kessar, and Bob Oliver, \emph{Fusion systems in
  algebra and topology}, London Mathematical Society Lecture Note Series, vol.
  391, Cambridge University Press, Cambridge, 2011. \MR{2848834}

\bibitem{AO}
Michael Aschbacher and Bob Oliver, \emph{Fusion systems}, Bull. Amer. Math.
  Soc. (N.S.) \textbf{53} (2016), no.~4, 555--615. \MR{3544261}

\bibitem{AS}
M.~F. Atiyah and G.~B. Segal, \emph{Equivariant {$K$}-theory and completion},
  J. Differential Geometry \textbf{3} (1969), 1--18. \MR{259946}

\bibitem{Be}
David~J. Benson, \emph{Cohomology of sporadic groups, finite loop spaces, and
  the {D}ickson invariants}, Geometry and cohomology in group theory ({D}urham,
  1994), London Math. Soc. Lecture Note Ser., vol. 252, Cambridge Univ. Press,
  Cambridge, 1998, pp.~10--23. \MR{1709949}

\bibitem{BK}
A.~K. Bousfield and D.~M. Kan, \emph{Homotopy limits, completions and
  localizations}, Lecture Notes in Mathematics, Vol. 304, Springer-Verlag,
  Berlin-New York, 1972. \MR{0365573}

\bibitem{Broto}
Carles Broto, \emph{Fusion systems in algebra and topology}, Gac. R. Soc. Mat.
  Esp. \textbf{21} (2018), no.~1, 183--202. \MR{3771029}

\bibitem{BLO1}
Carles Broto, Ran Levi, and Bob Oliver, \emph{The homotopy theory of fusion
  systems}, J. Amer. Math. Soc. \textbf{16} (2003), no.~4, 779--856.
  \MR{1992826}

\bibitem{BLO2}
\bysame, \emph{Discrete models for the {$p$}-local homotopy theory of compact
  {L}ie groups and {$p$}-compact groups}, Geom. Topol. \textbf{11} (2007),
  315--427. \MR{2302494}

\bibitem{Br}
Kenneth~S. Brown, \emph{Cohomology of groups}, Graduate Texts in Mathematics,
  vol.~87, Springer-Verlag, New York, 1994, Corrected reprint of the 1982
  original. \MR{1324339}

\bibitem{Can}
Jos\'e{} Cantarero, \emph{Una visi\'on local de los grupos finitos}, 2020,
  \url{https://www.youtube.com/playlist?list=PLlC_fV4djn7i8OeoB0gEl_l7HlFu3jvpx}.

\bibitem{CSV}
Jos\'e{} Cantarero, J\'er\^ome Scherer, and Antonio Viruel, \emph{Nilpotent
  {$p$}-local finite groups}, Ark. Mat. \textbf{52} (2014), no.~2, 203--225.
  \MR{3255138}

\bibitem{C}
Andrew Chermak, \emph{Fusion systems and localities}, Acta Math. \textbf{211}
  (2013), no.~1, 47--139. \MR{3118305}

\bibitem{Cr}
David~A. Craven, \emph{The theory of fusion systems}, Cambridge Studies in
  Advanced Mathematics, vol. 131, Cambridge University Press, Cambridge, 2011,
  An algebraic approach. \MR{2808319}

\bibitem{DW2}
W.~G. Dwyer and C.~W. Wilkerson, \emph{Homotopy fixed-point methods for {L}ie
  groups and finite loop spaces}, Ann. of Math. (2) \textbf{139} (1994), no.~2,
  395--442. \MR{1274096}

\bibitem{GAP4}
The GAP~Group, \emph{{GAP -- Groups, Algorithms, and Programming, Version
  4.13.1}}, 2024.

\bibitem{GL}
George Glauberman and Justin Lynd, \emph{Control of fixed points and existence
  and uniqueness of centric linking systems}, Invent. Math. \textbf{206}
  (2016), no.~2, 441--484. \MR{3570297}

\bibitem{Grodal}
Jesper Grodal, \emph{The classification of {$p$}-compact groups and homotopical
  group theory}, Proceedings of the {I}nternational {C}ongress of
  {M}athematicians. {V}olume {II}, Hindustan Book Agency, New Delhi, 2010,
  pp.~973--1001. \MR{2827828}

\bibitem{H}
Marshall Hall, Jr., \emph{The theory of groups}, The Macmillan Company, New
  York, 1959. \MR{103215}

\bibitem{He}
Ellen Henke, \emph{Recognizing {${\rm SL}_2(q)$} in fusion systems}, J. Group
  Theory \textbf{13} (2010), no.~5, 679--702. \MR{2720198}

\bibitem{I}
I.~Martin Isaacs, \emph{Finite group theory}, Graduate Studies in Mathematics,
  vol.~92, American Mathematical Society, Providence, RI, 2008. \MR{2426855}

\bibitem{Lo}
Ran Levi and Bob Oliver, \emph{Construction of 2-local finite groups of a type
  studied by {S}olomon and {B}enson}, Geom. Topol. \textbf{6} (2002), 917--990.
  \MR{1943386}

\bibitem{Li}
Markus Linckelmann, \emph{Introduction to fusion systems}, Group representation
  theory, EPFL Press, Lausanne, 2007, pp.~79--113. \MR{2336638}

\bibitem{MP}
John Martino and Stewart Priddy, \emph{Unstable homotopy classification of
  {$BG_p^{\wedge}$}}, Math. Proc. Cambridge Philos. Soc. \textbf{119} (1996),
  no.~1, 119--137. \MR{1356164}

\bibitem{Mil}
John Milnor, \emph{Construction of universal bundles. {II}}, Ann. of Math. (2)
  \textbf{63} (1956), 430--436. \MR{77932}

\bibitem{O1}
Bob Oliver, \emph{Equivalences of classifying spaces completed at odd primes},
  Math. Proc. Cambridge Philos. Soc. \textbf{137} (2004), no.~2, 321--347.
  \MR{2092063}

\bibitem{O2}
\bysame, \emph{Equivalences of classifying spaces completed at the prime two},
  Mem. Amer. Math. Soc. \textbf{180} (2006), no.~848, vi+102. \MR{2203209}

\bibitem{P}
Lluis Puig, \emph{Frobenius categories}, J. Algebra \textbf{303} (2006), no.~1,
  309--357. \MR{2253665}

\bibitem{Q}
Daniel Quillen, \emph{The spectrum of an equivariant cohomology ring. {I},
  {II}}, Ann. of Math. (2) \textbf{94} (1971), 549--572; ibid. (2) 94 (1971),
  573--602. \MR{298694}

\bibitem{R}
Joseph~J. Rotman, \emph{An introduction to the theory of groups}, fourth ed.,
  Graduate Texts in Mathematics, vol. 148, Springer-Verlag, New York, 1995.
  \MR{1307623}

\bibitem{RV}
Albert Ruiz and Antonio Viruel, \emph{The classification of {$p$}-local finite
  groups over the extraspecial group of order {$p^3$} and exponent {$p$}},
  Math. Z. \textbf{248} (2004), no.~1, 45--65. \MR{2092721}

\bibitem{S}
Ronald Solomon, \emph{Finite groups with {S}ylow {$2$}-subgroups of type
  {$3$}}, J. Algebra \textbf{28} (1974), 182--198. \MR{0344338}

\end{thebibliography}

\end{document}